\documentclass[11pt]{article}
\usepackage[english]{babel}
\usepackage{amsthm, amsmath, amssymb, amsbsy, amsfonts, latexsym, euscript}
\usepackage [latin1]{inputenc}
\usepackage{graphicx}
\theoremstyle{plain}
\newtheorem{theorem}{Theorem}[section]

\newtheorem{lemma}[theorem]{Lemma}
\newtheorem{corollary}[theorem]{Corollary}
\theoremstyle{definition}
\newtheorem{definition}[theorem]{Definition}

\begin{document}

\title{\bf {A  necessary condition for solvability by radicals}}

\author{Askold Khovanskii}

\maketitle
\begin{abstract}
This note was prepared as a handout for the MAT401 course ``Polynomial equations  and fields", taught at the University of Toronto  in Spring 2026.  It presents a proof of a necessary condition for the solvability of algebraic equations by radicals, based on Galois theory.
We begin with a brief overview of the relevant basic results from Galois theory, as covered in MAT401, and use -- without proof -- several standard (and relatively simple) results from the course textbook \cite{Rotman}.

The sufficient condition for solvability by radicals, which is based on  linear algebra, we will present in the next handout.
\end{abstract}

\section{Some basic theorems of Galois theory}

Our  presentation of a part of the Galois theory in MAT401 is based on the Degree  formula (see \cite{Rotman}, Lemma 49, page 53) and on  the following definitions and theorems:

\begin{definition} A finite field extension $K\subset E$ is {\it normal} if $E$ is a splitting field over $K$ of a separable polynomial $P\in K[x]$.
\end{definition}

\begin{definition} A group $G(E,K)$ of all automorphisms of $E$, which fix pointwise all elements of the field $K$, is called the {\it Galois group} of the normal extension $K\subset E$.
\end{definition}

\begin{theorem}\label{groupdegre} The Galois group $G$ of a normal field extension $K\subset E$  contains  $\# (G)=[E:K]$ elements.
\end{theorem}

Proof of Theorem \ref{groupdegre} can be found in our textbook \cite{Rotman} (see Theorem 56,   page 60). Theorem 56 is a direct corollary of Theorem 51 (page 56).

\begin{theorem} \label{lifting}Let  $E$ be the splitting field over $K$ of a separable polynomial $P\in K[x]$, and let $\sigma :K\to K$ be an automorphism that fixes all coefficients of the polynomial $P$. Then $\sigma$ can be lifted to an automorphism $\hat \sigma:E\to E$.
\end{theorem}

Theorem \ref{lifting} also is a direct corollary of Theorem 51 (page 56) of  our textbook \cite{Rotman}.

\section{Immediate corollaries of basic theorems}

Let us list immediate corollaries of the Degree formula and Theorems \ref{groupdegre},~\ref{lifting}.

\begin{lemma}\label{otherfield} Let $B$ be an intermediate field of a normal field extension $K\subset E$, i.e., the inclusions  $K\subset B\subset F$  hold.  Then  the field extension $B\subset E$ is normal.
 \end{lemma}

\begin{proof} If $E$ is the  splitting field of a separable polynomial $P(x)\in K[x]$ over $K$ and $B$ is an intermediate field, $K\subset B \subset E$, then $E$ is the splitting field of the same  separable polynomial $P(x)$ over $B$.
\end{proof}

\begin{corollary} The Galois group $G(E,K)$ of a normal field extension $K\subset E$
 fixes only elements of the field $K$.
\end{corollary}

\begin{proof} Let $B\subset E$ be a field of all pointwise  fixed elements under the action  of the group $G(E,K)$ on the field $E$. By Lemma \ref{otherfield}
and
by Theorem \ref{groupdegre},  we have $[E:B]=\#G$. By the same Theorem \ref{groupdegre}, we have $[E:K]=\#G$.
By the Degree formula, the following identity holds $[E:B] [B:K]=[E:K]$ which implies that $[B:K]=1$.  It implies  that $B=K$.
\end{proof}

\begin{definition} Let  $B$ be an  intermediate field of a normal field extension  $K\subset E$.
We denote by $G_B(E,K)$ a subgroup of the Galois group $G(E,K)$ consisting of all elements of the group $G(E,K)$ which fixes pointwise all elements  of the field $B$.
\end{definition}

 \begin{corollary}\label{subgroup}  The group $G_B(E,K)$ is naturally isomorphic to the Galois  group $G(E,B)$.
 \end{corollary}

\begin{proof} Indeed, each homomorphism $g:E\to E$ from the Galois group $G(E,K)$      is a lifting to the extension $E$ of the identity map from $K$ to $K$. It belongs to the group $G_B(E,K)$
if and only if it  fixes all elements of the  field $B$.  Such element $g$ is a lifting  of the identity isomorphism  of the field $B$ to the extension $E$.
 \end{proof}

\begin{corollary}\label{factorgroup} Let $B$ be an intermediate field of a normal field extension $K\subset E$. Assume that the extension $K\subset B$ is normal, i.e., the field $B$ is the splitting field of some separable polynomial $P\in K[x]$. Then the group  $G_B(E,K)$ is a normal subgroup of the Galois group $G(E,K)$, and the factor group $G(E,K)/G_B(E,K)$ is naturally isomorphic to the Galois group $G(B,K)$.
 \end{corollary}

\begin{proof} Since $B\subset E$ is the splitting  field of a polynomial over $K$, any automorphism  $g\in G(E,K)$ maps $B$ to itself. Restriction of homomorphisms $g\in G(E,K)$ to the intermediate field $B$ provides a homomorphism $\pi:G(E,K)\to G(B,K)$ of the Galois group $G(E,K)$ to the Galois group $G(B,K)$. The group $G_B(E,K)$ is the kernel of the homomorphism $\pi$. Thus, $G_B(E,K)$ is a normal subgroup of the Galois group $G(E,K)$. The homomorphism $\pi$ is surjective. Indeed, by Theorem \ref{lifting}, any
isomorphism $\sigma:B\to B$, which belongs to the Galois group $G(B,K)$, can be lifted to the isomorphism of the field $E$ over the field $K$, i.e., the isomorphism $\sigma$ belongs to the image of the group $G(E,K)$ under the homomorphism $\pi$. Thus, the factor group $G(E,K)/G(E,B)$ is isomorphic to the Galois group $G(B,K)$.
\end{proof}

\section{Another  basic theorem of Galois theory}

Let  $E$ be a splitting field  of a separable polynomial   over a field $K$. Since the extension $K\subset E$ has finite degree, any element $a\in E$ is algebraic over $K$. The Galois group $G(E,K)$ provides an explicit formula for the minimal polynomial of an element $a\in E$ over $K$.

\begin{definition} An {\it  orbit}  $O(a)$  of an element $a\in  E$ under the action of the Galois group $G(E,K)$ is the set of elements $\omega_i\in E$ representable in the form $\omega_i  = g_i(a)$, where $g_i$ is an element of the group $G(E,K)$. For an element $a\in E$ we define a {\it monic polynomial $Q_a(x)$} by the following formula:
\begin{equation}\label{formula}
Q_a(x)= \prod_{\omega_i\in O(a)}(x-\omega_i),
\end{equation}
whose the product is taken over all elements  $\omega_i$ of  the orbit $O(a)$.
\end{definition}

\begin{theorem} \label{alsep} The minimal polynomial of an element $a\in E$ is the polynomial $Q_a(x)$. All roots of this polynomial are simple, and all of them belong to the field $E$.
 \end{theorem}

\begin{proof} All coefficients of the monic   polynomial
$Q_a$(x) belong to the field $K$.  Indeed, any element $g\in G(E,K)$ provides a permutation of elements  in the orbit $O(a)$. Such permutation  does not change  coefficients of the polynomial $Q_a$, since they  are symmetric functions of the elements of the orbit $O(a)$. Since  coefficients of $Q_a$ are fixed under the action of the Galois group $G(E,k)$, they belong to the field $K$. By construction, all roots of the polynomial $Q_a(x)$ are simple and all of them  belong to the field $E$.

Let us show that $Q_a(x)$ is the minimal polynomial of the element $a$ over $K$. Indeed, any automorphism $\sigma\in G(E,K)$ fixes coefficients of the minimal polynomial. Thus, it sends the  root $a$ of the minimal polynomial to a root of the minimal polynomial. Thus, all elements of the orbit $O(a)$ are  roots of the minimal polynomial of $a$.

\end{proof}

 \section{Needed results of Galois theory. Summary }

Let us summarize the above results of Galois theory which are needed for proving the necessary condition of solvability of equations by radicals.

 \begin{enumerate}

\item  For  every  normal field extension $K\subset E$ its Galois group $G(E,K)$ is defined.

\item For every intermediate field $K\subset B\subset E$ of the normal field extension $K\subset E$  the extension $B\subset E$ is normal. Its Galois group $G(EB)$ is naturally isomorphic to the subgroup $G_B(E,K)$ of the Galois group $G(E,K)$ which contains all isomorphisms $g\in G(E,K)$
    such that $g(b)=b$ for every element  $b$ of intermediate   field $B$.

\item   Different intermediate fields  $B$ correspond to different subgroups $G_B(E,K)$: the  subgroup  $G_B(E,K)$ fixes only the  elements of the   intermediate field $B$.

 \item If an intermediate field $B$ is a splitting field of a separable polynomial over $K$, then the group $G_B(E,B)$ (which is isomorphic to the Galois group $ G(E,B)$)  is a normal subgroup in the Galois  $G(E,K)$. The factor group $G(E,K)/G_B(E,K)$ is isomorphic to the Galois group $G(B,K)$.
     
\item Moreover, Galois group $G(E,K)$ allows to provide an explicit formula for the minimal polynomial $Q_a(x)$  over $K$ of any element $a\in E$, see~(\ref{formula}).
\end{enumerate}

\section{Nested set of normal extensions}

We start with the following definition.

\begin{definition}  A set of  increasing   fields
\begin{equation}\label{fields}
K=E_0\subset E_1\subset\dots\subset E_n
\end{equation}
we call  a   {\it nested set  of normal extensions} of the ground field $K$ if for any  $i= 1,\dots, n$ the  field
$E_i$ is a normal extension of the ground field $K$.
\end{definition}

\begin{definition}
A decreasing chain of groups
\begin{equation}\label{groups}
G=G_0\supset G_1\dots\supset G_n=e
\end{equation}
is {\it associated} with the nested set of normal extensions (\ref{fields}) if the group
$G=G_0$ is  the Galois group $G(E_n,E_0)$ of the normal extension $E_0\subset E_n$,
and for $i=1,\dots, n$ the group  $G_i$ is the  subgroup  $G_{E_i} (E_n,E_0)$ of the group $G=G(E_n,E_0)$ which fixes all elements of  the field  $E_i$ pointwise.
\end{definition}

\begin{lemma}\label{chain} For any nested set of normal extensions  (\ref{fields}) and  for any $i= 0, \dots, n-1$ the  group   $G_{i+1}$ is a normal subgroup of the group $G_i$ from (\ref{groups}), and the factor group $G_i/G_{i+1}$ is isomorphic to the Galois group $G(E_{i+1}, E_i)$ of the normal extension $E_i\subset E_{i+1}$.
\end{lemma}

\begin{proof} For any $i=1,\dots,n-1$  the group $G_i$ is the Galois group of the normal extension $E_i\subset E_n$. Indeed, this statement follows from Corollary \ref {subgroup} applied to the intermediate field $E_i$  of the normal field extension extension  $K\subset E_n$.

 Now, the lemma follows from Corollary \ref{factorgroup} applied for the intermediate field $t E_{i+1}$ of the field extension extension  $E_i\subset E_n$, since the extension $E_i\subset E_{i+1}$ is normal.

 \end{proof}

\begin{corollary}\label{norm}  If, in the assumptions of Lemma\ref{chain}, for each $i=0,\dots,n-1$ the  Galois group of the normal field  extension $E_i\subset E_{i+1}$ is commutative, then the Galois group of the normal field extension  $K=E_0\subset E_n$ is solvable.
 \end{corollary}

 \begin{proof}  By Lemma \ref{chain}, the chain of groups (\ref{groups}), associated with the nested set of normal extensions for each $i=0,1,\dots,n-1$,  the group $G_{i+1}$ is the normal subgroup in $G_i$, and the factor group $G_i/G_{i+1}$ is the Galois group of the normal extension $E_i\subset E_{i+1}$ which, by the assumption, is commutative. Thus, the group $G$ is solvable.
 \end{proof} 

\section{Chains and nested sets of radical field extensions}

Starting from this section, we will assume, by default, that all fields we are dealing with have zero characteristic.

\begin{definition} A {\it  chain of  radical extensions} of a ground field $K$  is a chain  of field extensions
\begin{equation}\label{radicals}
K=R_0\subset R_1\subset \dots \subset R_n,
\end{equation}
such that for each $i=1,\dots,n$
the field $R_i$ is obtained from the field $R_{i-1}$ by adjoining  an element $a_i$,  satisfying a relation $a_i^{k_i}\in  R_{i-1}$, where $k_i$ is  a natural number which we will call  the  {\it $i$-th characteristic degree} of the radical chain (\ref{radicals}).
\end{definition}

\begin{definition} Let $P(x)\in K[x]$ be a polynomial over a ground  field $K$. The equation $P(x)=0$ over $K$ is {\it solvable by radicals} if there is a chain of radical extensions (\ref{radicals}) and a homomorphism  $\pi:E\to R_n$ of the splitting field $E$ of the polynomial $P$ over the ground field $K$, such that its restriction  to the ground  field $K$ is the identity map.
\end{definition}

Unfortunately, for a chain of  radical extensions (\ref{radicals}) the field extension $K\subset R_n$, in general, is not a normal extension and its Galois group is not defined.

Below, we  define a nested set of normal radical extensions (\ref{fields}) which is almost as simple as a chain of radical extensions, but for which the Galois group $G(E_n,K)$ is defined.

\begin{definition}  A nested set of normal field extensions (\ref{fields}) we will call  a   {\it nested set  of normal  radical extensions} if

\begin{enumerate}
\item the field $E_1$ is the splitting field over the  ground field $K=E_0$ of a polynomial $x^N-1$,  where $N$ is a natural number;
 \item for any $i=2,\dots,n$ the field extension $K\subset E_i$ is normal;
 \item  for any $i=2,\dots,n$  the field  $E_i$ is the splitting field over the field $E_{i-1}$ of a polynomial $Q_{i-1}(x)=\prod_{\omega _m\in O_{i-1}} (x^{k_i}-\omega_m)$,  where $O_{i-1}$ is a finite subset of the field  $E_{i-1}$ and $k_i$ is a natural number, which divides the number $N$ from 1.
\end{enumerate}
\end{definition}

\section{Normalization of a chain of radical extensions}

\begin{definition} A nested set of normal radical extensions
\begin{equation}\label{nestedrad}
 K=E_0\subset K_u=E_1\subset E_2\subset \dots\subset E_{n+}
 \end{equation}
we will call a {\it normalization} of  a given chain of radical extension (\ref{radicals}) if the following conditions hold:
 \begin{enumerate}
\item the field $K_u=E_1$ is the splitting field of the polynomial $x^N-1$ over $K$, where $N$ is the least common multiple of the set $\{k_i\}$  of characteristic degrees $k_i$ of the radical chain (\ref{radicals});
\item for $i=1,\ldots, n$ the inclusion $R_i\subset E_{i+1}$ holds.
\end{enumerate}
\end{definition}

\begin{theorem}\label{exorm} For any chain of radical extensions of a ground field $K$ one can construct its normalization.
\end{theorem}

\begin{proof}
We construct the nested set of normal radical extensions inductively. The field $K_u=E_1$ is already defined: it is the splitting field of the polynomial $x^N-1$ over the ground field $K$, where $N$  is the least common multiple of the set of the characteristic degrees $k_i$ of the radical chain (\ref{radicals}).

Assume that we already constructed a field $E_{i}$ which is a normal extension of the ground field $K$ and contains the field $R_{i-1}$. By definition, the field $R_i$ is obtained from the field $R_{i-1}$ by adjoining a element $a_i$ such that the element $a^{k_i}_i=b$ belongs to the field $R_{i-1}\subset E_i$.

By induction, the field $E_i$ is a normal extension of the field $K$. By Theorem \ref{alsep}, the minimal polynomial $Q_b(x)$ of the element $b\in E_{i-1}$ is given by  formula $Q_b(x)= \prod_{\omega_m\in O(b)}(x-\omega_m)$, where the product is taken over the orbit $O(b)$ of the element $b\in E_i$ under the action of the Galois group $G(E_i,K)$. We define the field $E_{i+1}$ as the splitting field of the polynomial $Q_b(x^{k_i})= \prod_{\omega_m\in O(b)}(x^{k_i}-\omega_m)$ whose coefficients belong to the ground field $K$. Thus, the field extension $K\subset E_{i+1}$ is normal. On the other hand, the field $E_{+1}$ is obtained from the field $E_i$ by adjoining all solutions  of equations $x^{k_i}=\omega_m$, where the elements $\omega_m$ belong to orbit $O(b)\subset E_i$. Thus, the field $E_{i+1}$ satisfies both conditions 1) and 2). The inductive construction is completed.
\end{proof}

\section{Galois groups of radical field extensions}

\begin{lemma}\label{addmul} Consider the ring $\Bbb Z/ (n)$ of integers modulo $n$ for some  $n>1$. The group of all automorphisms of the additive group $\Bbb Z_n$ of this ring   is isomorphic to its multiplicative group $U(n)$.
\end{lemma}

\begin{proof} An automorphism $g$ of the  group $\Bbb Z_n$ sends the generator $(1\mod n)$ of the group to another its generator. Thus, $g(1)\equiv m \mod n$, where $m$ is relatively prime number with $n$, which could be considered as an element of the group $U(n)$. Since $g$ is an automorphism of $\Bbb Z_n$, we have  $g(k)\equiv mk \mod n$. Thus, $g$  acts on the group  $\Bbb Z_n$ by multiplication on an invertible element $m$. On the other hand,  the multiplication on any invertible element $m$  provides an automorphism  of the  group $\Bbb Z_n$.
\end{proof}

\begin{lemma}\label{uroots}  Let $E$ be a splitting field over a ground field $K$ of a polynomial $x^n-1$, where $n>1$  is a natural number. Then the Galois group $G(E,K)$ is isomorphic to a subgroup of  $U(n)$. In particular, the group $G(E,K)$ is commutative.
\end{lemma}

\begin{proof} Roots of the polynomial $x^n-1$ in the splitting field $E$ form a multiplicative group which is isomorphic to the additive cyclic group $\Bbb Z_n$. Each element of the Galois group $G(E,K)$ provides an automorphism of the group $\Bbb Z_n$. Thus, we have a natural homomorphism $\pi:G(E,K)\to U(n)$. The homomorphism $\pi$ has no kernel  since  roots of  the polynomial $x^n-1$ generate the field $E$ over the field $K$.
 Thus, the group $G(E,K)$ is isomorphic to a subgroup of the group $U(n)$(which is equal to the image of the group $G(E,K)$ under the homomorphism $\pi$);
\end{proof}

\begin{lemma}\label{rroot} Let $E$ be a splitting field over a ground field $K$ of a polynomial $x^n-a$, where $n>1$  is a natural number and $a\in K$. Assume that $K$ contains all roots of unity of degree $n$.  Then the Galois group $G(E,K)$ is isomorphic to a subgroup of  the cyclic additive group $\Bbb Z_n$. In particular, the group $G(E,K)$ is commutative.
\end{lemma}

\begin{proof} Let $x_0$ be a root of the equation $x^n-a=0$. Each  root of this equation can be represented as $x_0 \omega^k$, where $\omega$ is a generator of the multiplicative group of  roots of unity of degree $n$, and the integer $k$ is defined up to addition of a number divisible by $n$.
Thus, one can identify the set of roots of the polynomial $x^n-a$ with the set of integral numbers modulo $n$.

If an element $g\in G(E,K)$ sends
the root $x_0$ to a root $x_0\omega^m$, then it sends a root $x_0 \omega^k$ to the root $x_0\omega^k \omega^m=x_0\omega^{m+k}$ for every integer $k$.  Indeed, by the assumption,    $\omega\in K$, so  we have  $g(\omega)=\omega$ and $g(\omega^k)=\omega^k$.   Thus, under the identification of the set of roots of the polynomial $x^n-a$ with integers modulo $n$, the  action of the element $g$ is  reduced to the addition of the element  $(m\mod n)$ to elements  $k\in \Bbb Z_n$. We obtain a map from the Galois group $G(E,K)$ to the group $\Bbb Z_n$ which sends an element $g$ to an element $m$. One can see that this map is a group homomorphism. This homomorphism has no kernel since the field $E$ is generated over the field $K$ by roots of the equation $x^n-a=0$.
Thus, the homomorphism embeds the Galois group $G(E,K)$ to the group $\Bbb Z_n$.

\end{proof}

\begin{lemma}\label{rroots} Let $E$ be a splitting field over a ground field $K$ of a polynomial $P(x)=(x^n-a_1)\cdot \ldots \cdot (x^n-a_m)$, where $n>1$  is a natural number and $a_1,\dots,a_m$ are different elements of the ground field $K$. Assume that $K$ contains all roots of unity of degree $n$.  Then the Galois group $G(E,K)$ is isomorphic to a subgroup of  the direct sum of $m$ copes $\Bbb Z_n(1),\dots, \Bbb Z_n(m)$ of the additive cyclic  group $\Bbb Z_n$. In particular, the group $G(E,K)$ is commutative.
\end{lemma}

\begin{proof} The set of all roots of the polynomial $P(x)$ can be represented as the union $\bigcup_{1\leq i\leq m}T_i$ of the sets $T_i$ of all roots of the polynomials $x^n-a_i$. The sets $T_i$  do not intersect each other, and each of them  contains $n$ elements. The group $G(E,K)$ maps each set $T_i$ to itself, since the set $T_i$ is the set of all roots of the polynomial $x^n-a_i$ over $K$.

Repeating the proof of Lemma \ref{rroot}, one can define a group homomorphisms $\pi_i:G(E,K)\to \Bbb Z_n(i)$ which describes the action of the Galois group on the set $T_i$ of  roots of the polynomial $x^n-a_i$. One can define the homomorphism  $\pi:G(E,K)\to \oplus _{1\leq i\leq m}\Bbb Z_n(i)$ which sends an element $g\in G(E,K)$ to the element $(\pi_1(g),\dots,\pi_m(g))$. The homomorphism $\pi$  has no kernel, since the field $E$ is generated over $K$ by roots of the polynomial $(x^n-a_1)\cdot \ldots \cdot(x^n-a_m)$.
Thus, the homomorphism $\pi$ embeds the Galois group $G(E,K)$ to the direct sum of groups $\Bbb Z_n(i)$.
\end{proof}

 \section{Necessary condition for solvability by radicals}

\begin{theorem} If a polynomial equation $P(x)=0$ over a field $K$ of characteristic zero is solvable by radicals, and   $E$ is the splitting field of the polynomial $P$, then the Galois group $G(EK)$ is solvable.
\end{theorem}

\begin{proof} If the equation $P=0$ is solvable by radicals, then there is a chain of radical extension $K=R_0\subset\dots\subset R_n$ such that $E\subset R_n$.
By Theorem \ref{exorm}, there exists a normalization $K=E_0\subset K_u=E_1\subset E_2\subset \ldots\subset E_{n+1}$ of the chain of radical extensions for which the inclusions $E_{n+1}\supset R_n\supset E$ hold.  By Lemma \ref{uroots}, the Galois group $G (K_u,K)$ is commutative. By Lemma \ref{rroots}, the Galois groups $G(E_2,E_1),\dots, G(E_{n+1}E_n)$ are commutative. Thus, by  Corollary \ref{norm},  the Galois group  $D=G(E_{n+1},K)$ is solvable.
The splitting field $E$ of the polynomial $P(x)$ is a normal intermediate field  the extension $K\subset E_{n+1}$. By Corollary \ref{factorgroup},  the Galois group $G(E,K)$ is a factor group of the group $G=G(E_{n+1},K)$. Thus, the group $G(E,K)$ is solvable.

 \end{proof}

\end{document}